\documentclass[a4paper,11pt]{article}
\usepackage{graphicx}
\usepackage{amsfonts}
\usepackage{amsmath}
\usepackage{hyperref}
\usepackage{xcolor}

\title{Orientable and negative orientable sequences}
\author{Chris J. Mitchell and Peter R. Wild
\\Information Security Group, Royal Holloway, University of London\\
\href{mailto:me@chrismitchell.net}{me@chrismitchell.net};
~~~~\href{mailto:peterrwild@gmail.com}{peterrwild@gmail.com}}
\date{14th February 2025}
\newtheorem{lemma}{Lemma}[section]
\newtheorem{theorem}[lemma]{Theorem}
\newtheorem{corollary}[lemma]{Corollary}
\newtheorem{definition}[lemma]{Definition}
\newtheorem{remark}[lemma]{Remark}
\newtheorem{construction}[lemma]{Construction}
\newtheorem{result}[lemma]{Result}
\newtheorem{note}[lemma]{Note}

\newenvironment{proof}[1][Proof]{\begin{trivlist}
\item[\hskip \labelsep {\bfseries #1}]}{\end{trivlist}}
\newcommand{\qed}{\nobreak \ifvmode \relax \else
      \ifdim\lastskip<1.5em \hskip-\lastskip
      \hskip1.5em plus0em minus0.5em \fi \nobreak
      \vrule height0.75em width0.5em depth0.25em\fi}
\begin{document}

\maketitle

\begin{abstract}
Analogously to de Bruijn sequences, orientable sequences have application in automatic
position-location applications and, until recently, studies of these sequences focused on the
binary case. In recent work by Alhakim et al., a range of methods of construction were described
for orientable sequences over arbitrary finite alphabets; some of these methods involve using
negative orientable sequences as a building block. In this paper we describe three techniques for
generating such negative orientable sequences, as well as upper bounds on their period.  We then go
on to show how these negative orientable sequences can be used to generate orientable sequences for
every non-binary alphabet size and for every tuple length. In doing so we use two closely related approaches described by
Alhakim et al. The periods of both negative orientable and orientable sequences that we construct are  of the same
order of magnitude as the upper bounds.

\end{abstract}

\section{Introduction}  \label{section:introduction}

Orientable sequences are periodic sequences with elements drawn from a finite alphabet with the
property that any subsequence of $n$ consecutive elements (an $n$-tuple) occurs at most once
\emph{in either direction}; they were introduced in 1992~\cite{Burns92,Burns93,Dai93}. For example,
the sequence over $\{0,1,2\}$ with period $[012]$ is orientable as the $2$-tuples it contains are
$(0,1),(1,2),(2,0)$ in the forward direction and $(0,2),(2,1),(1,0)$ in the reverse direction. They
are of interest due to their application in certain position-resolution scenarios. For the binary
case, a construction and an upper bound on the period were established by Dai et al.~\cite{Dai93},
and further constructions were established by Mitchell and Wild~\cite{Mitchell22} and Gabri\'{c}
and Sawada~\cite{Gabric24,Gabric24b}. A bound on the period and methods of construction for $q$-ary
alphabet sequences (for arbitrary $q$) were given by Alhakim et al.~\cite{Alhakim24a}.

In this paper we use two of the methods proposed in~\cite{Alhakim24a} to construct $q$-ary
orientable sequences for arbitrary $n$ and $q$. This provides orientable sequences with period
of the same order of magnitude as the upper bound.  Both these methods involve the use of
$q$-ary negative orientable sequences as a building block.  Negative orientable sequences,
first introduced in~\cite{Alhakim24a}, are analogous to orientable sequences except that any
$n$-tuple occurs at most once in either a period of the sequence or a period of the negative of
the reverse sequence.  Interpreting the alphabet of the example above as residues modulo $3$,
the example above is not negative orientable as the negatives of the reverse $2$-tuples are
$(0,1),(1,2),(2,0)$, equal to the forward $2$-tuples. The sequence, over the same alphabet,
with period $[011]$ is negative orientable (but not orientable) as the forward $2$-tuples are
$(0,1),(1,1),(1,0)$ and the negatives of the reverse $2$-tuples are $(2,0),(2,2),(0,2)$.

We first demonstrate methods for generating negative orientable sequences for every $q>2$ and $n$.  All these sequences can in
turn be used to generate orientable sequences via the inverse Lempel Homomorphism,
following~\cite{Alhakim24a}. This parallels recent work by Gabri\'{c} and Sawada~\cite{Gabric24}
who also constructed $q$-ary orientable sequences for arbitrary $n$ and $q$, albeit using a different approach.

The main motivation in studying negative orientable sequences has been to help construct new
examples of orientable sequences.  However, they may have applications of their own, and
potentially merit further study.

The remainder of this paper is structured as follows.  In the rest of
Section~\ref{section:introduction}, basic concepts are introduced.
Section~\ref{section:NOS_construction_for_n 2} describes how maximal length orientable
sequences can be constructed for every alphabet size $q>2$ when the window size is 2.
Section~\ref{section:NOS} establishes an upper bound on the period of negative orientable
sequences and gives three methods of construction for such sequences. In
Sections~\ref{section:OS_near_optimal} and \ref{section:goodOS}, we show how the negative
orientable sequences constructed in Section~\ref{section:NOS} can be used to construct new
orientable sequences.  Finally, Section~\ref{section:Conclusions} concludes the paper, and
suggests directions for future work.

\subsection{Basic terminology}

We first establish some simple notation, following~\cite{Alhakim24a}. For mathematical convenience
we consider the elements of a sequence to be elements of $\mathbb{Z}_q$ for an arbitrary integer
$q>1$.

For a sequence $S = (s_i)$ we write $\mathbf{s}_n(i) = (s_i,s_{i+1},\ldots,s_{i+n-1})$.  Since we
are interested in tuples occurring either forwards or backwards in a sequence we also introduce the
notion of a reversed tuple, so that if $\mathbf{u} = (u_0,u_1,\ldots,u_{n-1})$ is a $q$-ary
$n$-tuple (a string of symbols of length $n$) then $\mathbf{u}^R = (u_{n-1},u_{n-2}, \ldots,u_0)$
is its \emph{reverse}.  The \emph{negative} of a $q$-ary $n$-tuple
$\mathbf{u}=(u_0,u_1,\ldots,u_{n-1})$ is the $n$-tuple $-\mathbf{u}=(-u_0,-u_1,\ldots,-u_{n-1})$.

We can then give the following.

\begin{definition}[\cite{Alhakim24a}]
A $q$-ary \emph{$n$-window sequence $S = (s_i)$} is a periodic sequence of elements from
$\mathbb{Z}_q$ ($q>1$, $n>1$) with the property that no $n$-tuple appears more than once in a
period of the sequence, i.e.\ with the property that if $\mathbf{s}_n(i) = \mathbf{s}_n(j)$ for
some $i,j$, then $i \equiv j \pmod m$ where $m$ is the period of the sequence.
\end{definition}

\begin{definition}[\cite{Alhakim24a}]
A $q$-ary $n$-window sequence $S = (s_i)$ is said to be an \emph{orientable sequence of order $n$}
(an $\mathcal{OS}_q(n)$) if $s_n(i) \neq s_n(j)^R$, for any $i,j$.
\end{definition}

A range of examples of orientable sequences can be found in Gabri\'{c} and Sawada~\cite{Gabric24}.
We also need two related definitions.

\begin{definition}[\cite{Alhakim24a}]  \label{definition:NOS}
A $q$-ary $n$-window sequence $S=(s_i)$ is said to be a \emph{negative orientable sequence of order
$n$} (a $\mathcal{NOS}_q(n)$) if $\mathbf{s}_n(i)\not=-{\mathbf{s}_n(j)}^R$, for any $i,j$.
\end{definition}

Examples of negative orientable sequences are given in Section~\ref{section:NOS-examples}.
As discussed in Alhalkim et al.~\cite{Alhakim24a}, it turns out that negative orientable sequences
can be used to construct orientable sequences. Observe that a sequence is orientable if and only if
it is negative orientable for the case $q=2$. Also note that if $S=(s_i)$ is orientable or negative
orientable then so is its negative ($-s_i$).

\subsection{Related work}

This paper builds on the work of Alhakim et al.~\cite{Alhakim24a}, in which recursive methods of
construction for non-binary orientable sequences are described. In particular, methods for
developing sequences suitable for use as `starter sequences' in the recursive constructions are
developed.

The work described here is complementary to parallel recent work by Gabri\'{c} and
Sawada~\cite{Gabric25}.

\section{Constructing maximal orientable sequences}  \label{section:NOS_construction_for_n 2}

We start by showing that, unlike the binary case, orientable sequences for $n=2$ with period
meeting the bound of~\cite{Alhakim24a} can easily be constructed for every $q>2$. Analogous results
have recently been obtained by Gabri\'{c} and Sawada~\cite{Gabric25}.

The bound of Theorem 4.11 of~\cite{Alhakim24a} is as follows:
Suppose that $S=(s_i)$ is an $\mathcal{OS}_q(n)$ ($q\geq2$, $n\geq 2$). Then the period of $S$ is
at most
\begin{eqnarray*}
(q^n - q^{\lceil n/2\rceil} - q^{\lceil (n-1)/2\rceil} + q)/2 & \mbox{if $q$ is odd}, \\
(q^n - q^{\lceil n/2\rceil} - q)/2 & \mbox{if $q$ is even}.
\end{eqnarray*}
Further, if $q$ is odd and $n\geq 6$ then the period of $S$ is at most
\begin{eqnarray*}
(q^n - 2q^{n/2} - q^{(n-2)/2} + 2q)/2 & \mbox{if $n$ is even}, \\
(q^n - q^{(n+1)/2} - 2q^{(n-1)/2} + q + q^2)/2 & \mbox{if $n$ is odd}.
\end{eqnarray*}

We note that all these bounds are asymptotically of the same order as $q^n/2$.

\begin{definition}
Let $S$ be an $OS_q(n)$. If $S$ has period meeting the bound of Theorem 4.11 of Alhakim et
al.~\cite{Alhakim24a} then $S$ is said to be \emph{maximal}.
\end{definition}

\begin{lemma}  \label{lemma:n=2}
There exists a maximal $OS_q(2)$ for all $q \ge 3$.
\end{lemma}

\begin{proof}
First observe that the ring sequence corresponding to an $OS_q(2)$ of period $p$ corresponds to the
list of vertices visited in a circuit of length $p$ in the complete graph $K_q$ with its $q$
vertices labelled $0,1,\ldots,q-1$.

If $q$ is odd, then, by \cite[Theorem 4.11]{Alhakim24a}, a maximal $OS_q(2)$ $S$ has period
$q(q-1)/2$, i.e.\ it corresponds to an Eulerian circuit in the complete graph on $q$ vertices,
$K_q$ (since $K_q$ has $q(q-1)/2$ edges). Every vertex of $K_q$ has degree $q-1$, which is even
since $q$ is odd, and hence such a circuit always exists by Euler's Theorem (see, for example,
Corollary 6.1 of Gibbons \cite{Gibbons85}), and the result follows.

If $q$ is even, then, by \cite[Theorem 4.11]{Alhakim24a}, a maximal $OS_q(2)$ $S$ has period
$q(q-2)/2$. Let $K^*_q$ be the complete graph with an arbitrary 1-factor removed; for example, if
the vertices are labelled $0,1,\ldots,q-1$, we could remove the edges ($2i$, $2i+1$) for every $i$,
$0\leq i<q/2$. It is simple to observe that $K^*_q$ has $q(q-2)/2$ edges and every vertex has
degree $q-2$, which is even. The desired result again follows from Euler's Theorem.  \qed
\end{proof}

\begin{remark}
There are simple algorithms for finding Eulerian circuits --- see for example Gibbons~\cite[Figure
6.5]{Gibbons85}. Moreover, a direct construction of an $OS_q(2)$ for the case where $q>2$ is a
prime is presented in~\cite[Construction 5.3]{Alhakim24a}.
\end{remark}

\section{Negative orientable sequences}  \label{section:NOS}

We next establish some fundamental results on negative orientable sequences (see
Definition~\ref{definition:NOS}), given their importance in constructing orientable sequences using
the methods of~\cite{Alhakim24a}.  We further describe three methods of construction for such
sequences.  Here and throughout the remainder of the paper we only consider the case $q>2$ for two
main reasons.  As observed in Section~\ref{section:introduction}, when $q=2$, negative orientable
sequences are the same as orientable sequences.  Further, the methods of construction we present in
this section do not work in the case $q=2$.

\subsection{A bound on the period}

We start by giving a bound on the period of negative orientable sequences.  To establish this bound
we need to introduce the following terminology to distinguish a special class of $n$-tuples. A
$q$-ary $n$-tuple $\mathbf{u}=(u_0,u_1,\dots,u_{n-1})$ is said to be \emph{negasymmetric} if
$u_i=-u_{n-1-i}$ for every $i$ ($0\leq i\leq n-1$), i.e.\ if $\mathbf{u}=-\mathbf{u}^R$\@.  Clearly
an $\mathcal{NOS}_q(n)$ cannot contain any negasymmetric $n$-tuples.  This motivates the following.

\begin{lemma}  \label{lemma:negasymmetric}
For $n\geq 2$, the number of $q$-ary negasymmetric $n$-tuples is
\begin{align*}
q^{n/2} & \mbox{~~if $n$ is even} \\
q^{(n-1)/2} & \mbox{~~if $n$ is odd and $q$ is odd} \\
2q^{(n-1)/2} & \mbox{~~if $n$ is odd and $q$ is even.}
\end{align*}
\end{lemma}

\begin{proof}
The $n$ even case is immediate. For $n$ odd ($n=2m+1$ say), an $n$-tuple
$(u_0,u_1,\ldots,u_{m-1},u_m,-u_{m-1},\ldots,-u_1,-u_0)$ is negasymmetric if and only if
$u_m=-u_m$. If $q$ is odd this implies $u_m=0$, and if $q$ is even then this implies $u_m$ is
either $0$ or $q/2$.  The result follows. \qed
\end{proof}

This gives the following simple bound.

\begin{corollary}  \label{corollary:simple-NOS-bound}
For $n\geq 2$, the period of an $\mathcal{NOS}_q(n)$ is at most
\begin{align*}
(q^n-q^{n/2})/2 & \mbox{~~if $n$ is even} \\
(q^n-q^{(n-1)/2})/2 & \mbox{~~if $n$ is odd and $q$ is odd} \\
(q^n-2q^{(n-1)/2})/2 & \mbox{~~if $n$ is odd and $q$ is even.}
\end{align*}
\end{corollary}

\begin{proof}
There are trivially $q^n$ $q$-ary $n$-tuples, of which only those that are non-negasymmetric can
occur in a $\mathcal{NOS}_q(n)$. If $\mathbf{u}$ is such an $n$-tuple, then at most one of
$\mathbf{u}$ and $-\mathbf{u}^R$ can occur in an $\mathcal{NOS}_q(n)$. The result follows. \qed
\end{proof}

As we next show, it is possible to improve on this simple bound.  We first need the following.

\begin{definition}
An $n$-tuple $\mathbf{u} = (u_0,u_1,\ldots,u_{n-1})$, $u_i \in \mathbb{Z}_q$ ($0\leq i\leq n-1)$),
is said to be $m$-\emph{negasymmetric} for some $m\leq n$ if and only if the $m$-tuple
$(u_0,u_1,\ldots,u_{m-1})$ is negasymmetric.
\end{definition}

Clearly an $n$-tuple is $n$-negasymmetric if and only if is it is negasymmetric.  We also need the
notion of uniformity.

\begin{definition}
An $n$-tuple $\mathbf{u}=(u_0,u_1,\ldots,u_{n-1})$, $u_i\in\mathbb{Z}_q$ ($0\leq i\leq n-1$), is
$c$-\emph{uniform} for some $c\in\mathbb{Z}_q$ if and only if $u_i=c$ for every $i$ ($0\leq i\leq
n-1$).
\end{definition}

We can then state the following elementary results.

\begin{lemma} \label{lemma:negasymmetry}
If $n\geq2$ and $\mathbf{u}=(u_0,u_1,\ldots,u_{n-1})$ is a $q$-ary $n$-tuple that is both
negasymmetric and $(n-1)$-negasymmetric, then $\mathbf{u}$ is $c$-uniform where $c$ is either 0 or
$q/2$ (the latter only applying if $q$ is even).
\end{lemma}

\begin{proof} Choose any $i$, $0\leq i\leq n-2$. Then, by negasymmetry we know that $u_{i+1}=-u_{n-2-i}$
(observing that $i\leq n-2$).  Also, by $(n-1)$-negasymmetry we know that $u_i=-u_{n-2-i}$, and
hence we have $u_i=u_{i+1}$.  The result follows. \qed
\end{proof}

\begin{definition}
Let $N^*_q(n)$ be the set of all non-negasymmetric $q$-ary $n$-tuples.
\end{definition}

The size of $N^*_q(n)$ follows from Lemma~\ref{lemma:negasymmetric}. As a first step towards
establishing our bound we need to define a special set of $n$-tuples, as follows.

\begin{definition}[\cite{Alhakim24a}]
Suppose $n\geq 2$, and that $\mathbf{v}=(v_0,v_1,\ldots,v_{n-2})$ is a $q$-ary $(n-1)$-tuple.  Then
let $L(\mathbf{v})$ be the following set of $q$-ary $n$-tuples:
\[  L(\mathbf{v}) = \{ \mathbf{u}=(u_0,u_1,\ldots,u_{n-1}):~
u_i=v_i,~~0\leq i\leq n-2 \}. \]
\end{definition}

That is, $L(\mathbf{v})$ is simply the set of $n$-tuples whose first $n-1$ entries equal
$\mathbf{v}$.  Clearly, the sets $L(\mathbf{v})$ for all $(n-1)$-tuples $\mathbf{v}$ are disjoint.
We are interested in how the sets $L(\mathbf{v})$ intersect with the set of $n$-tuples occurring in
either $S$ or $-S^R$ when $S$ is a $\mathcal{NOS}_q(n)$ and $\mathbf{v}$ is a negasymmetric
$(n-1)$-tuple.

\begin{definition}
Suppose $n\geq 2$, $r\geq 1$, $S=(s_i)$ is a $\mathcal{NOS}_q(n)$, and
$\mathbf{v}=(v_0,v_1,\ldots,v_{n-2})$ is a $q$-ary $(n-1)$-tuple.  Then let
\[ L^*_S(\mathbf{v})=\{\mathbf{u}\in L(\mathbf{v}):~
 \mathbf{u} \text{~appears in~} S \text{~or~} {-}S^R \}.\]
\end{definition}

It is important to observe that, from the definition of negative orientable, an element of
$L^*_S(\mathbf{v})$, as with any $n$-tuple, can only appear at most once in $S$ or $-S^R$.   We now
state a result on which our bound is built.

\begin{lemma} \label{lemma:even}
Suppose $n\geq 2$, $S=(s_i)$ is a $\mathcal{NOS}_q(n)$, and $\mathbf{v}=(v_0,v_1,\ldots,v_{n-2})$
is a $q$-ary negasymmetric $(n-1)$-tuple.  Then $|L^*_S(\mathbf{v})|$ is even.
\end{lemma}

\begin{proof}
Suppose $L^*_S(\mathbf{v})$ is non-empty.  Then, by definition, and (without loss of generality)
assuming that an element of $L(\mathbf{v})$ occurs in $S$ (as opposed to $-S^R$), we know that
$v_i=s_{j+i}$ for some $j$ ($0\leq i\leq n-2$).  Since $\mathbf{v}$ is negasymmetric, it follows
immediately that $v_i=-s_{j+n-2-i}$ ($0\leq i\leq n-2$).  That is, for each occurrence of an
element of $L(\mathbf{v})$ in $S$, there is an occurrence of a, necessarily distinct, element of
$L(\mathbf{v})$ in $-S^R$, and vice versa.  The result follows. \qed
\end{proof}

If $|L(\mathbf{v})|$ is odd, Lemma~\ref{lemma:even} immediately shows that $S$ and $-S^R$ combined
must omit at least one of the $n$-tuples in $L(\mathbf{v})$.  We can now state our main result.

\begin{theorem} \label{theorem:NOS-bounds}
Suppose that $S=(s_i)$ is a $\mathcal{NOS}_q(n)$ ($q\geq2$,
$n\geq 2$). Then the period of $S$ is at most
\begin{align*}
(q^n - q^{\lfloor n/2\rfloor} - q^{\lfloor (n-1)/2\rfloor} + 1)/2 & ~~\mbox{if $q$ is odd}, \\
(q^n - 2q^{(n-1)/2})/2 - 1 & ~~\mbox{if $q$ is even and $n$ is odd}, \\
(q^n - q^{n/2})/2 - 1 & ~~\mbox{if $q$ is even and $n$ is even}.
\end{align*}

\end{theorem}

\begin{proof}

The proof involves considering the set $N^*_q(n)$ of non-negasymmetric $q$-ary $n$-tuples, and
showing that the tuples in certain disjoint subsets of this set cannot all occur in a
$\mathcal{NOS}_q(n)$. We divide our discussion into two cases depending on the parity of $q$.

\begin{itemize}
\item First suppose that $q$ is odd.  Suppose also that $\mathbf{v}$ is a negasymmetric
    non-uniform $(n-1)$-tuple. From Lemma~\ref{lemma:negasymmetric} there are
    $q^{\lfloor(n-1)/2\rfloor}$ negasymmetric $(n-1)$-tuples of which precisely one, namely the
    all-zero tuple, is uniform --- since $q$ is odd; i.e.\ there are
    $q^{\lfloor(n-1)/2\rfloor}-1$ such $\mathbf{v}$. From Lemma~\ref{lemma:negasymmetry}, none
    of the $q$ elements of $L(\mathbf{v})$ are negasymmetric, and hence $|L(\mathbf{v})\cap
    N^*_q(n)|=q$.  Since $|L(\mathbf{v})|=q$ is odd, Lemma~\ref{lemma:even} implies that
    $L^*_S(\mathbf{v})$ must omit at least one of the $n$-tuples in $L(\mathbf{v})$.  That is,
    there are $q^{\lfloor(n-1)/2\rfloor}-1$ non-negasymmetric $n$-tuples which cannot appear in
    $S$ or $-S$.  The bound for $q$ odd follows from
    Corollary~\ref{corollary:simple-NOS-bound}.

\item Now suppose $q$ is even. Let $\mathbf{v}$ be a $c$-uniform $(n-1)$-tuple for $c=0$ or
    $c=q/2$. There are clearly two such $(n-1)$-tuples. By Lemma~\ref{lemma:negasymmetry}
    precisely one of the elements of $L^*_S(\mathbf{v})$ is negasymmetric, namely the relevant
    $c$-uniform $n$-tuple, and hence $|L(\mathbf{v})\cap N^*_q(n)|=q-1$. Now, since $q$ is
    even, Lemma~\ref{lemma:even} implies that at least one element of $L(\mathbf{v})\cap
    N^*_q(n)$ cannot be contained in $L^*_S(\mathbf{v})$. The result for $q$ even follows from
    Corollary~\ref{corollary:simple-NOS-bound}. \qed
\end{itemize}
\end{proof}

We note that, as in the case of orientable sequences, all these bounds on negative orientable sequences are asymptotically of the same order as $q^n/2$.

We conclude by tabulating the values of the bounds of Theorem~\ref{theorem:NOS-bounds} for small
$q$ and $n$.
\begin{table}[htb]
\caption{Bounds on the period of a $\mathcal{NOS}_q(n)$ (from Theorem~\ref{theorem:NOS-bounds})}
\label{table_bounds}
\begin{center}
\begin{tabular}{crrrrrrr} \hline
 ~    & $n=2$ & $n=3$ & $n=4$ & $n=5$ & $n=6$ & $n=7$ & $n=8$ \\ \hline
$q=2$ &    0  &     1 &     5 &    11 &    27 &     55 &     119 \\
$q=3$ &    3  &    11 &    35 &   113 &   347 &   1067 &    3227 \\
$q=4$ &    5  &    27 &   119 &   495 &  2015 &   8127 &   32639 \\
$q=5$ &   10  &    58 &   298 &  1538 &  7738 &  38938 &  194938 \\
$q=6$ &   14  &   101 &   629 &  3851 & 23219 & 139751 &  839159 \\
$q=7$ &   21  &   165 &  1173 &  8355 & 58629 & 411429 & 2881029 \\
$q=8$ &   27  &   247 &  2015 & 16319 &130815 &1048063 & 8386559 \\
\hline
\end{tabular}
\end{center}
\end{table}

For the case $n=2$, the bound simplifies to $(q^2-q)/2$ for $q$ odd, and $(q^2-q)/2-1$ for $q$
even. As we next show, in this case the bound is tight.


\subsection{Construction I: Sequences with \texorpdfstring{$n=2$}{n=2}}  \label{section:NOSq2}

The first method of construction we present provides negative orientable sequences of maximal
length for every $q>3$ in the case $n=2$.  The method involves joining certain explicitly defined
cycles.

Consider the directed graph $G_q$ with vertex set $\mathbb Z_q$, $q \ge 3$, and all directed arcs
$(x,y)$, $x,y \in \mathbb Z_q$, except for $(x,-x)$, $x \in \mathbb Z_q$. Then a
$\mathcal{NOS}_q(2)$ $S$ corresponds to a circuit in $G_q$. Moreover $-S^R$ is also such a circuit
and these two circuits are arc-disjoint. We say that a circuit $C=[c_0,c_1,\dots,c_{m-1}]$ of
length $m$, consisting of arcs $(c_i,c_{i+1})$, $i=0,\dots,m-1$, identifying $c_m$ with $c_0$, in
$G_q$ is \emph{nice} if it is arc-disjoint from $-C^R$. We note that if $C'$ is another nice
circuit in $G_q$ that is arc-disjoint from $C$ and $-C^R$ and if it shares a vertex with $C$ then
$C$ and $C'$ may be joined to make a nice circuit of length the sum of the lengths of $C$ and $C'$.

Suppose that $q$ is odd. Consider the following circuits in $G_q$.

\begin{align*}
C_0 & = \left[0110220\dots 0\frac{q-1}{2}\frac{q-1}{2}\right] \text{~~of length~} 3\frac{q-1}{2} \\
C_1 & = \left[121(-2)131(-3)1\dots 1\frac{q-1}{2}1(-\frac{q-1}{2})\right] \text{~~of length~} 4\frac{q-3}{2} \\
    & \vdots \\
C_i & = \left[i(i+1)i-(i+1)i\dots i\frac{q-1}{2}i-\frac{q-1}{2}\right] \text{~~of length~}
4\frac{q-i-2}{2} \\
    & \vdots \\
C_{\frac{q-3}{2}}& = \left[\frac{q-3}{2}\frac{q-1}{2}\frac{q-3}{2}-\frac{q-1}{2}\right] \text{~~of length 4}.
\end{align*}

These circuits are arc-disjoint since by construction the arcs of $C_i$ involve the vertex $i$,
with the exception of the loops $(j,j)$ in $C_0$, in such a way that there is no repetition.
Moreover they are arc-disjoint to $-C_i^R$, $i=0,\dots,\frac{q-3}{2}$ since by construction the
arcs of $-C_i^R$ involve the vertices $-j$ in such a way that there is no repetition. Thus they are
mutually nice circuits and it may be seen that these circuits may be joined to give a
$\mathcal{NOS}_q(2)$ $C$. A simple calculation shows that $C$ has period $\frac{q(q-1)}{2}$ and
therefore is maximal by Theorem~\ref{theorem:NOS-bounds}.

Suppose that $q$ is even. Consider the following circuits in $G_q$.

\begin{align*}
C_0 & =\left[0110220\dots 0\frac{q-2}{2}\frac{q-2}{2}\right] \text{~~of length~} 3\frac{q-2}{2} \\
C_1 & =\left[121(-2)131(-3)1\dots 1\frac{q-2}{2}1(-\frac{q-2}{2})1\frac{q}{2}\right] \text{~~of length~}
4\frac{q-4}{2}+2 \\
 & \vdots \\
C_i & =\left[i(i+1)i-(i+1)i\dots i\frac{q-2}{2}i-\frac{q-2}{2}i\frac{q}{2}\right] \text{~~of length~} 4\frac{q-2i-2}{2}+2 \\
 & \vdots \\
C_{\frac{q-4}{2}} & =\left[\frac{q-4}{2}\frac{q-2}{2}\frac{q-4}{2}-\frac{q-2}{2}\frac{q-4}{2}\frac{q}{2}\right]
\text{~~of length~} 6 \\
C_{\frac{q-2}{2}} & =\left[\frac{q-2}{2}\frac{q}{2}\right] \text{~~of length~} 2.
\end{align*}

By a similar analysis it can be seen that these circuits may be joined to give a
$\mathcal{NOS}_q(2)$ $C$. A straightforward calculation shows that $C$ has period
$\frac{q(q-1)}{2}-1$ and therefore is maximal, again by Theorem~\ref{theorem:NOS-bounds}.

These examples establish the following Lemma by construction.

\begin{lemma}  \label{lemma:NOSn=2}
There exists a maximal $\mathcal{NOS}_q(2)$ for all $q \ge 3$.
\end{lemma}

\begin{note}  \label{note:constant2}
Observe that the maximal $\mathcal{NOS}_q(2)$ constructed using the approaches just described all
possess the $i$-uniform 2-tuples $(i,i)$ for $1\leq i<q/2$.
\end{note}

\subsection{Construction II:  A construction for general \texorpdfstring{$n$}{n}}  \label{section:NOS-generaln}

The second method of construction involves the notion of the pseudoweight of an $n$-tuple, which is
equal to the sum of the values in a tuple after changing all the zeros to $q/2$.  The construction
then builds on the observation that if the pseudoweights of all the tuples in a sequence are less
than half the maximum possible value ($nq/2$) then the pseudoweight of all the $n$-tuples in the
negative sequence will be greater than half the maximum possible value.

We first need the following formal definition of pseudoweight.

\begin{definition}  \label{definition:pseudoweight}
Suppose $\mathbf{u}=(u_0,u_1,\ldots,u_{n-1})$ is an $n$-tuple of elements of $\mathbb{Z}_q$
($q>1$). Define the function $f:\mathbb{Z}_q\rightarrow\mathbb{Q}$ as follows: for any
$u\in\mathbb{Z}_q$ treat $u$ as an integer in the range $[0,q-1]$ and set $f(u)=u$ if $u\neq0$ and
$f(u)=q/2$ if $u=0$. Then the \emph{pseudoweight} of $\mathbf{u}$ is defined to be the sum
\[ w^*(\mathbf{u}) = \sum_{i=0}^{n-1}f(u_i) \]
where the sum is computed in $\mathbb{Q}$.
\end{definition}

As a simple example for $q=3$, the 4-tuple $(0,1,1,2)$ has weight $0+1+1+2=4$ and pseudoweight
$1.5+1+1+2=5.5$, since $f(0)=\frac{3}{2}$.

\begin{theorem}
Suppose $S$ is a $q$-ary $n$-window sequence ($n\geq 2$) with the property that all the
$n$-tuples appearing in $S$ have pseudoweight less than $nq/2$.  Then $S$ is a $\mathcal{NOS}_q(n)$.
\end{theorem}

\begin{proof}
Consider any $n$-tuple $\mathbf{u}$ occurring in $S$, and by definition we know that
$w^*(\mathbf{u})<nq/2$. We claim that $w^*(-\mathbf{u}^R)=nq-w^*(\mathbf{u})$. This follows
immediately from the definition of $f$ since $f(-u_i)=q-f(u_i)$ for every possible value of $u_i$.

Hence, since $w^*(\mathbf{u})<nq/2$, it follows immediately that
$w^*(-\mathbf{u}^R)>nq/2$.  Thus the $n$-tuples in $-S^R$ are all distinct from the $n$-tuples in
$S$. Moreover, the $n$-tuples in $S$ (and in $-S^R$) are all distinct because $S$ is an $n$-window
sequence. Hence $S$ is a $\mathcal{NOS}_q(n)$. \qed
\end{proof}

It remains to establish how to construct the desired sequences $S$.  Suppose, for $q>2$ and
$n\geq2$, $G$ is a directed graph with vertices the $q$-ary ${(n-1)}$-tuples with pseudoweight less
than $nq/2-1$, and with a directed edge connecting vertices $\mathbf{u}=(u_0,u_1,\ldots,u_{n-2})$
and $\mathbf{v}=(v_0,v_1,\ldots,v_{n-2})$ if and only if $u_i=v_{i-1}$, $(1\leq i\leq n-2)$, and
$w^*(\mathbf{u})+f(v_{n-2})<nq/2$, where $f$ is as in Definition~\ref{definition:pseudoweight}.
As is conventional, we identify the edge
connecting $(u_0,u_1,\ldots,u_{n-2})$ to $(u_1,u_2,\ldots,u_{n-2},v)$ with the $n$-tuple
$(u_0,u_1,\ldots,u_{n-2},v)$, and $G$ contains as an edge every $q$-ary $n$-tuple with pseudoweight
less than $nq/2$.

A directed circuit in $G$ will clearly give us a sequence $S$ with the desired property. Consider
any vertex $\mathbf{u}=(u_0,u_1,\ldots,u_{n-2})$. An incoming edge
\[ (s,u_0,u_1,\ldots,u_{n-2}) \]
must satisfy $s+w^*(\mathbf{u})<nq/2$. Similarly an outgoing edge
\[(u_0,u_1,\ldots,u_{n-2},t)\]
must satisfy $t+w^*(\mathbf{u})<nq/2$.  That is, the in-degree of every vertex is the same as its
out-degree. Moreover, $G$ is clearly connected, as it is straightforward to construct a `low
pseudoweight' directed path between any pair of vertices --- that is, for any vertex
$\mathbf{u}=(u_0,u_1,\ldots,u_{n-2})$ there is a path to $(1,1,\ldots,1)$ via vertices of the form
$(u_i,u_{i+1},\ldots,u_{n-2},1,1,\ldots,1)$ and similarly from $(1,1,\ldots,1)$ to any vertex
$\mathbf{v}$. This means there exists a directed Eulerian Circuit in $G$ (see, for example,
\cite[Corollary 6.1]{Gibbons85}), yielding a $q$-ary $n$-window sequence $S$ with the property that
all the $n$-tuples appearing in $S$ have weight less than $nq/2$, and with period equal to the
number of $q$-ary $n$-tuples with pseudoweight less than $nq/2$.

Let $r_{q,n,s}$ ($q\geq2$, $n\geq 1$) denote the number of $q$-ary $n$-tuples with pseudoweight
exactly $s$, where $r_{q,n,s}=0$ by definition if $s<n$ or $s>n(q-1)$. Then the above discussion
establishes the following result.

\begin{lemma}  \label{lemma:NOS-generaln}
There exists a $\mathcal{NOS}_q(n)$ of period $\frac{q^n-r_{q,n,nq/2}}{2}$ for all $q>2$ and $n\geq
2$.
\end{lemma}

\begin{proof}
There are trivially $q^n$ $q$-ary $n$-tuples, and for every $n$-tuple $\mathbf u$ of
pseudoweight less than  $nq/2$ there is a corresponding `pseudonegative'
$n$-tuple $-\mathbf u$ of weight greater than $nq/2$.  The result follows.  \qed
\end{proof}

\subsection{Enumerating \texorpdfstring{$n$}{n}-tuples with given pseudoweight}

Lemma~\ref{lemma:NOS-generaln} means that it is of interest to know the value of $r_{q,n,s}$. If
$q$ is even then $r_{q,n,s}$ is only defined for integer values $s$, and if $q$ is odd it is only
defined when $s$ is a multiple of 0.5. We have the following elementary result.

\begin{lemma}  \label{lemma:rqns}
Suppose $q\geq 2$ and $n\geq 1$.
\begin{enumerate}
\item[(i)] If $q$ is even, $r_{q,1,i}=1$ for $1\leq i\leq q-1$, $i\neq q/2$, and
    $r_{q,1,q/2}=2$.
\item[(ii)] If $q$ is odd, $r_{q,1,i}=1$ for $1\leq i\leq q-1$, and $r_{q,1,q/2}=1$.
\item[(iii)] If $q$ is even, $\sum_{i=n}^{n(q-1)}r_{q,n,i} = q^n$.
\item[(iv)] If $q$ is odd, $\sum_{j=0}^{2n(q-2)}r_{q,n,n+j/2} = q^n$.
\item[(v)] If $n>1$, $r_{q,n,s} = \sum_{i=1}^{q-1}r_{q,n-1,s-i}+r_{q,n-1,s-q/2}$.
\end{enumerate}

\end{lemma}

\begin{proof}
Parts (i)--(iv) are trivially true (observing that the 1-tuples $(0)$ and $(q/2)$ both have
pseudoweight $q/2$). Part (v) follows immediately by considering the $q$ possible ways of adding a
final entry to an $(n-1)$-tuple to make an $n$-tuple of a given pseudoweight, observing that
inserting $0$ as the final entry will increase the pseudoweight by $q/2$. \qed
\end{proof}

We can also make some simple observations.

\begin{corollary}  \label{corollary:rqns}
Suppose $q\geq2$.
\begin{enumerate}
\item[(i)] If $q$ is odd and $2\leq w\leq 2q-2$, then $r_{q,2,q}=q$ and $r_{q,2,w}<q$ if $w\neq
    q$.
\item[(ii)] If $q$ is even and $2\leq w\leq 2q-2$, then $r_{q,2,q}=q+2$ and $r_{q,2,w}\leq q$
    if $w\neq q$.
\item[(iii)] If $q$ is odd, $n\geq 3$ and $n\leq w\leq n(q-1)$, then $r_{q,n,w}<q^{n-1}$.
\item[(iv)] If $q$ is even, $n\geq 3$ and $n\leq w\leq n(q-1)$, then $r_{q,n,w}\leq
    q^{n-3}(q^2+2)$.
\item[(v)] If $n\geq 1$ and $0\leq i\leq 2n$ then $r_{3,n,n+i/2}=N_2(i,n)$, where $N_2(i,n)$ is
    a trinomial coefficient (see Section~\ref{section:enumerating_weight} below).
\item[(vi)] If $1\leq n\leq w\leq 3n$, then $r_{4,n,w}=\binom{2n}{w-n}$.
\end{enumerate}
\end{corollary}

\begin{proof}
\begin{itemize}

\item[(i)] If $q$ is odd, then, by Lemma~\ref{lemma:rqns}(v),
    \[ r_{q,2,q} = \sum_{i=1}^{q-1}r_{q,1,q-i} + r_{q,1,q-q/2} = (q-1)+1, \mbox{(by Lemma~\ref{lemma:rqns}(ii))}. \]
    If $w\neq q$, then, by the same result, $r_{q,2,w}$ will be the sum of $q$ terms at most
$q-1$ of which are non-zero, all equal to 1, i.e.\ $r_{q,2,w}\leq q-1$.
\item[(ii)] If $q$ is even then, by Lemma~\ref{lemma:rqns}(v),
    \[ r_{q,2,q} = \sum_{i=1}^{q-1}r_{q,1,i}+r_{q,1,q/2} = q + 2, \mbox{(by Lemma~\ref{lemma:rqns}(i))}. \]
If $w\neq q$, then, by the same result, $r_{q,2,w}$ is equal to the sum of $q$ terms at most
$q-1$ of  which are non-zero, with at most one equal to $2$ and all others no greater than $2$
($r_{q,n-1,w-q/2}$  is at most one since $w\neq q$).  Hence $r_{q,2,w}\leq (q-2)+2=q$.
\item[(iii)] This follows by induction on $n$. If $n=3$ then, by Lemma~\ref{lemma:rqns}(v),
    $r_{q,3,w}$ is the sum of $q$ terms, all of which are less than $q$ except possibly at most
    one  equal to $q$ --- from (i).  Hence, since $q>2$, $r_{q,3,w}\leq (q-1)^2+q<q^2$,
    and the claim holds. The induction step is trivial since $r_{q,m+1,w}$ is equal to the sum
    of $q$ terms $r_{q,m,i}$, each of which is less than $q^{m-1}$.
\item[(iv)] This also follows by induction on $n$. If $n=3$ then, by Lemma~\ref{lemma:rqns}(v),
    $r_{q,3,w}$ is the sum of $q$ terms, all of which are at most $q$ except possibly at most one
     equal to $q+2$) --- from (ii).  Hence, $r_{q,3,w}\leq (q-1)q+(q+2)=q^2+2$, and the
    claim holds. The induction step is trivial since $r_{q,m+1,w}$ is equal to the sum of $q$
    terms $r_{q,m,i}$, each of which is at most $q^{m-3}(q^2+2)$.

\item[(v)] When $q=3$ the recursion of Lemma~\ref{lemma:rqns}(v) is $r_{3,n,s} =
    r_{3,n-1,s-2}+r_{3,n-1,s-3/2}+r_{3,n-1,s-1}$ and by Lemma~\ref{lemma:rqns}(ii) we have
    $r_{3,1,1}=r_{3,1,3/2}=r_{3,1,2}=1$. It follows that $r_{3,n,n+i/2}$ is the coefficient of
    $x^i$ in the expansion of $(1+x+x^2)^n$, i.e.\ the trinomial coefficient $N_2(i,n)$.
\item[(vi)] When $q=4$ the recursion of Lemma~\ref{lemma:rqns}(v) is $r_{4,n,s} =
    r_{4,n-1,s-3}+2r_{4,n-1,s-2}+r_{3,n-1,s-1}$ and by Lemma~\ref{lemma:rqns}(i) we have
    $r_{4,1,1}=1,r_{4,1,2}=2,r_{4,1,3}=1$. It follows that $r_{4,n,w}$ is the coefficient of
    $x^{w-n}$ in the expansion of $(1+2x+x^2)^n=(1+x)^{2n}$, i.e.\ the binomial coefficient
    $\binom{2n}{w-n}$. \qed
\end{itemize}
\end{proof}

\begin{note}  \label{note:rqns}
Lemma~\ref{lemma:NOS-generaln} when combined with Corollary~\ref{corollary:rqns}(i) implies that,
in the case $n=2$, the $\mathcal{NOS}_q(2)$ generated using the approach described in
Section~\ref{section:NOS-generaln} has optimal period.

Also, it follows from Corollary~\ref{corollary:rqns}(iii) and (iv) that a $\mathcal{NOS}_q(n)$
generated using the approach described in Section~\ref{section:NOS-generaln} has, by Lemma~\ref{lemma:NOS-generaln}, period at least $(q^n-q^{n-3}(q^2+2))/2=\frac{q^3-q^2-2}{2q^3} q^n$. Since $\frac{q^3-q^2-2}{q^3}$ approaches $1$ as $q\rightarrow \infty$, these negative orientable sequences have period  the same order of magnitude as the upper bound.  Of
course, if $n>2$ the period is unlikely to actually be optimal, as the construction avoids all
$n$-tuples of pseudoweight exactly $nq/2$.
\end{note}
The values of $r_{q,n,nq/2}$ for $q=3$, $q=4$ and $1\leq n\leq 6$ are given in
Table~\ref{table_r_q34}.  These values were derived using Corollary~\ref{corollary:rqns}(v) and
(vi).

\begin{table}[htb]
\caption{Values of $r_{q,n,nq/2}$ for small $q$ and $n$} \label{table_r_q34}
\begin{center}
\begin{tabular}{crrrrrrr} \hline
 ~    & $n=2$ & $n=3$ & $n=4$ & $n=5$ & $n=6$ & $n=7$ & $n=8$ \\ \hline
$q=3$ &    3  &     7 &    19 &    51 &   141 &    393 &    1107 \\
$q=4$ &    6  &    20 &    70 &   252 &   924 &   3432 &   12870 \\
$q=5$ &    5 &     13 &    69 &   261 &  1301 &   5685 &   27525 \\
$q=6$ &    8 &     38 &   196 &  1052 &  5774 &  32146 &  180772 \\
$q=7$ &    7 &     19 &   183 &   791 &  6613 &  35561 &  267639 \\
$q=8$ &   10 &     62 &   426 &  3052 & 22354 & 166014 & 1245066 \\
  \hline
\end{tabular}
\end{center}
\end{table}

The periods of the negative orientable sequences using Construction II for $q=3$, $q=4$ and $1\leq
n\leq 6$ are given in Table~\ref{table_NOS_periods_II_q34}.  These values were derived using
Lemma~\ref{lemma:NOS-generaln} and Table~\ref{table_r_q34}.

\begin{table}[htb]
\caption{$\mathcal{NOS}_q(n)$ periods (Construction II) for small $q$ and $n$}
\label{table_NOS_periods_II_q34}
\begin{center}
\begin{tabular}{crrrrrrr} \hline
 ~    & $n=2$ & $n=3$ & $n=4$ & $n=5$ & $n=6$ & $n=7$ & $n=8$ \\ \hline
$q=3$ &    3  &    10 &    31 &    96 &   294 &    897 &    2727 \\
$q=4$ &    5  &    22 &    93 &   386 &  1586 &   6476 &   26333 \\
$q=5$ &   10 &     56 &   278 &  1432 &  7162 &  36220 &  181550 \\
$q=6$ &   14 &     89 &   550 &  3362 & 20441 & 123895 &  749422 \\
$q=7$ &   21 &    162 &  1109 &  8008 & 55518 & 393991 & 2748581 \\
$q=8$ &   27 &    225 &  1835 & 14858 &119895 & 965569 & 7766075 \\
\hline
\end{tabular}
\end{center}
\end{table}

These values can be compared with the bounds given in Table~\ref{table_bounds}.

\subsection{Examples of Construction II}  \label{section:NOS-examples}

We next give some simple examples of the general $n$ construction.

\subsubsection{A negative orientable sequence for \texorpdfstring{$n=2$}{n=2}
and \texorpdfstring{$q=3$}{q=3}}

The nine 3-ary 2-tuples and their pseudoweights are shown in
Table~\ref{table:tuples-n2q3}.

\begin{table}[htb]
\caption{Pseudoweights of $q$-ary $n$-tuples: $n=2$, $q=3$}
\label{table:tuples-n2q3}
\begin{center}
\begin{tabular}{|l|c||l|c||l|c|}
\hline
$\mathbf{u}$ & $w^*(\mathbf{u})$ & $\mathbf{u}$ & $w^*(\mathbf{u})$ & $\mathbf{u}$ & $w^*(\mathbf{u})$  \\ \hline
11 & 2   & 00 & 3   & 02 & 3.5   \\ \hline
01 & 2.5 & 12 & 3   & 20 & 3.5   \\ \hline
10 & 2.5 & 21 & 3   & 22 & 4     \\ \hline
\end{tabular}
\end{center}
\end{table}

There are clearly just three 2-tuples of pseudoweight less than $nq/2=3$, and
these can be joined to form the ring sequence
\[ S_{3,2}=[110]. \]
This is an optimal period $\mathcal{NOS}_3(2)$ (see Table~\ref{table_bounds}),
with $w_3(S_{3,2})=2$ (a unit in $\mathbb{Z}_3)$).

\subsubsection{A negative orientable sequence for \texorpdfstring{$n=3$}{n=3}
and \texorpdfstring{$q=3$}{q=3}}

The 27 3-ary 3-tuples and their pseudoweights are shown in
Table~\ref{table:tuples-n3q3}.

\begin{table}[htb]
\caption{Pseudoweights of $q$-ary $n$-tuples: $n=3$, $q=3$}
\label{table:tuples-n3q3}
\begin{center}
\begin{tabular}{|l|c||l|c||l|c|}
\hline
$\mathbf{u}$ & $w^*(\mathbf{u})$ & $\mathbf{u}$ & $w^*(\mathbf{u})$ & $\mathbf{u}$ & $w^*(\mathbf{u})$  \\ \hline
111 & 3   & 211 & 4   & 020 & 5     \\ \hline
011 & 3.5 & 000 & 4.5 & 200 & 5     \\ \hline
101 & 3.5 & 012 & 4.5 & 122 & 5     \\ \hline
110 & 3.5 & 021 & 4.5 & 212 & 5     \\ \hline
001 & 4   & 102 & 4.5 & 221 & 5     \\ \hline
010 & 4   & 120 & 4.5 & 022 & 5.5   \\ \hline
100 & 4   & 201 & 4.5 & 202 & 5.5   \\ \hline
112 & 4   & 210 & 4.5 & 220 & 5.5   \\ \hline
121 & 4   & 002 & 5   & 222 & 6     \\ \hline
\end{tabular}
\end{center}
\end{table}

There are ten 3-tuples of pseudoweight less than $nq/2=4.5$, and these can be
joined to form the ring sequence
\[ S_{3,3} = [0100112111]. \]
This is a $\mathcal{NOS}_3(3)$ with period one less than the bound of
Theorem~\ref{theorem:NOS-bounds} (see Table~\ref{table_bounds}) with
$w_3(S_{3,3})=2$ (a unit in $\mathbb{Z}_3)$.

\subsubsection{A negative orientable sequence for \texorpdfstring{$n=3$}{n=3}
and \texorpdfstring{$q=4$}{q=4}}

The 64 4-ary 3-tuples and their pseudoweights are shown in
Table~\ref{table:tuples-n3q4}.

\begin{table}[htb]
\caption{Pseudoweights of $q$-ary $n$-tuples: $n=3$, $q=4$}
\label{table:tuples-n3q4}
\begin{center}
\begin{tabular}{|l|c||l|c||l|c||l|c|}
\hline
$\mathbf{u}$ & $w^*(\mathbf{u})$ & $\mathbf{u}$ & $w^*(\mathbf{u})$ & $\mathbf{u}$ & $w^*(\mathbf{u})$ & $\mathbf{u}$ & $w^*(\mathbf{u})$  \\ \hline
 111 & 3   & 113 & 5   & 022 & 6   & 230 & 7     \\ \hline
 011 & 4   & 131 & 5   & 202 & 6   & 302 & 7     \\ \hline
 101 & 4   & 311 & 5   & 220 & 6   & 320 & 7     \\ \hline
 110 & 4   & 122 & 5   & 123 & 6   & 133 & 7     \\ \hline

 112 & 4   & 212 & 5   & 132 & 6   & 313 & 7     \\ \hline
 121 & 4   & 221 & 5   & 213 & 6   & 331 & 7   \\ \hline
 211 & 4   & 000 & 6   & 231 & 6   & 223 & 7   \\ \hline
 001 & 5   & 002 & 6   & 312 & 6   & 232 & 7   \\ \hline

 010 & 5   & 020 & 6   & 321 & 6   & 322 & 7    \\ \hline
 100 & 5   & 200 & 6   & 222 & 6   & 033 & 8     \\ \hline
 012 & 5   & 013 & 6   & 003 & 7   & 303 & 8     \\ \hline
 021 & 5   & 031 & 6   & 030 & 7   & 330 & 8     \\ \hline

 102 & 5   & 103 & 6   & 300 & 7   & 233 & 8     \\ \hline
 120 & 5   & 130 & 6   & 023 & 7   & 323 & 8     \\ \hline
 201 & 5   & 301 & 6   & 032 & 7   & 332 & 8     \\ \hline
 210 & 5   & 310 & 6   & 203 & 7   & 333 & 9     \\ \hline
\end{tabular}
\end{center}
\end{table}

There are 22 3-tuples of pseudoweight less than $nq/2=6$, and these can be
joined to form the ring sequence
\[ S_{4,3}=[0122~1201~0210~0113~1112~11]. \]
This is a $\mathcal{NOS}_4(3)$ with period five less than the bound of
Theorem~\ref{theorem:NOS-bounds} (see Table~\ref{table_bounds}) with
$w_4(S_{4,3})=0$.

\subsection{Construction III: A construction avoiding zeros}  \label{section:NOS-construction-3}

We now give a further method of construction for negative orientable sequences. This approach
yields sequences with a somewhat shorter period than those generated by Construction II. However,
as discussed in Section~\ref{section:goodOS}, the sequences have a special property that makes them
suitable for use with a recursive method of construction for orientable sequences described in
\cite{Alhakim24a}.

We first need the following definition.

\begin{definition}
If $\mathbf{u}=(u_0,u_1,\ldots,u_{n-1})$ is an $n$-tuple of elements of $\mathbb{Z}_q$ ($q>1$),
then the \emph{weight} of $\mathbf{u}$ is defined to be the sum
\[ w(\mathbf{u}) = \sum_{i=0}^{n-1}u_i \]
where we compute the sum in $\mathbb{Z}$ and treat the values $u_i$ as integers in the range
$[0,q-1]$.
\end{definition}

We can now state the following elementary result.

\begin{theorem} \label{theorem:zero-free}
Suppose $n\geq 2$, $q>2$, and $S$ is a $q$-ary $n$-window sequence containing no zeros and with the
property that all the $n$-tuples appearing in $S$ have weight less than $nq/2$.  Then $S$ is a
$\mathcal{NOS}_q(n)$.
\end{theorem}

\begin{proof}
If $\mathbf{u}$ is an $n$-tuple in $S$ then $w(-\mathbf{u}^R)=nq-w(\mathbf{u})$ since $\mathbf{u}$
is zero-free. Hence, since $w(\mathbf{u})<nq/2$, it follows immediately that
$w(-\mathbf{u}^R)>nq/2$.  Hence the $n$-tuples in $-S^R$ are all distinct from the $n$-tuples in
$S$. Moreover, the $n$-tuples in $S$ (and in $-S^R$) are all distinct because $S$ is an $n$-window
sequence. Hence $S$ is a $\mathcal{NOS}_q(n)$. \qed
\end{proof}

It remains to establish how to construct sequences $S$ with the properties of
Theorem~\ref{theorem:zero-free}. Suppose, for $q>2$ and $n\geq2$, $G$ is a directed graph with
vertices the zero-free $q$-ary $(n-1)$-tuples with weight less than $nq/2-1$, and with a directed
edge connecting vertices $\mathbf{u}=(u_0,u_1,\ldots,u_{n-2})$ and
$\mathbf{v}=(v_0,v_1,\ldots,v_{n-2})$ if and only if $u_i=v_{i-1}$, $(1\leq i\leq n-2)$, and
$w(\mathbf{u})+v_{n-2}<nq/2$.

As previously, we identify the edge connecting $(u_0,u_1,\ldots,u_{n-2})$ to
$(u_1,u_2,\ldots,u_{n-2},v)$ with the $n$-tuple $(u_0,u_1,\ldots,u_{n-2},v)$. It follows
immediately that every zero-free $q$-ary $n$-tuple with weight less than $nq/2$ will appear as an
edge in $G$.

A directed circuit in $G$ will clearly give us a sequence $S$ with the desired property. Consider
any vertex $\mathbf{u}=(u_0,u_1,\ldots,u_{n-2})$. An incoming edge
\[ (s,u_0,u_1,\ldots,u_{n-2}) \]
must satisfy $s+w(\mathbf{u})<nq/2$. Similarly an outgoing edge
\[(u_0,u_1,\ldots,u_{n-2},t)\]
must satisfy $t+w(\mathbf{u})<nq/2$.  That is, the in-degree of every vertex is the same as its
out-degree. Moreover, $G$ is clearly connected, as it is straightforward to construct a `low
weight' directed path between any pair of vertices (just as in the discussion in
Section~\ref{section:NOS-generaln}). This means there exists a directed Eulerian Circuit in $G$
(see, for example, \cite[Corollary 6.1]{Gibbons85}), yielding a $q$-ary zero-free $n$-window
sequence $S$ with the property that all the $n$-tuples appearing in $S$ have weight less than
$nq/2$, and with period equal to the number of zero-free $q$-ary $n$-tuples with weight less than
$nq/2$.

If $nq$ is odd (i.e.\ if $n$ and $q$ are both odd), then the number of zero-free $n$-tuples with
weight less than $nq/2$ is simply $(q-1)^n/2$ (since for every zero-free $n$-tuple of weight less
than $nq/2$ there is one of weight greater than $nq/2$). If $nq$ is even then the number of
$n$-tuples is $((q-1)^n-k_{q,n,nq/2})/2$, where $k_{q,n,w}$ is the number of zero-free $q$-ary
$n$-tuples with weight exactly $w$.

The above discussion establishes the following result.

\begin{lemma}  \label{lemma:NOS-generaln-zerofree}
There exists a $\mathcal{NOS}_q(n)$ containing no zeros of period $(q-1)^n/2$ (if $q$ and $n$ are both odd) or
$((q-1)^n-k_{q,n,nq/2})/2$ (if $q$ or $n$ is even) for all $q>2$ and $n\geq 2$.
\end{lemma}

We show, in the next subsection, that for $q>2$, $k_{q,n,nq/2}<(q-1)^{n-1}$.   Thus a sequence of $\mathcal{NOS}_q(n)$ exists with period at least $\frac{q-2}{q-1} (q-1)^n/2$. Although this period is not of the same order of magnitude as the upper bound of Theorem~\ref{theorem:NOS-bounds} we note that the sequences constructed in this section are negative orientable sequences of order $n$ over the alphabet $ \mathbb Z_q \backslash \{0\}$  of size $q-1$ and the period of such sequences is bounded above by $(q-1)^n/2$.  Now $\frac{q-2}{q-1}\ge 1/2$ and approaches $1$ as $q\rightarrow\infty$ so the period is the same order of magnitude as this upper bound. Indeed when $nq$ is odd this upper bound is attained.

As we discuss below, it is of interest to know the weight of the $\mathcal{NOS}_q(n)$ constructed
by the above method. We have the following elementary result.

\begin{lemma}  \label{lemma:NOS-generaln-weight}
Suppose $S$ is a $\mathcal{NOS}_q(n)$ constructed according to Theorem~\ref{theorem:zero-free}
using a sequence $S$ containing all zero-free $q$-ary $n$-tuples with weight less than $nq/2$.
Then:
\[ w(S) = \frac{1}{n}\sum_{w=n}^{\lfloor nq/2 \rfloor}wk_{q,n,w}. \]
\end{lemma}

\begin{proof}
If we add the weights of all the $n$-tuples occurring in the ring sequence $S$ then the result will
clearly be $nw(S)$. Since $S$ contains all possible zero-free $q$-ary $n$-tuples of weight less
than $nq/2$, the result follows. \qed
\end{proof}
The values of $k_{q,n,nq/2}$ for $q=3$, $q=4$ and $1\leq n\leq 6$ are given in
Table~\ref{table_k_q34} (observing that $k_{q,n,nq/2}$ is undefined for $q$ and $n$ both odd).
These values were derived using Lemma~\ref{lemma:knqw}.

\begin{table}[htb]
\caption{Values of $k_{q,n,nq/2}$ for $q=3$ and $q=4$} \label{table_k_q34}
\begin{center}
\begin{tabular}{crrrrr} \hline
 ~    & $n=2$ & $n=3$ & $n=4$ & $n=5$ & $n=6$ \\ \hline
$q=3$ &    2  &    -- &     6 &    -- &    20 \\
$q=4$ &    3  &     7 &    19 &    51 &   141 \\  \hline
\end{tabular}
\end{center}
\end{table}

The periods of the negative orientable sequences using Construction III for $q=3$, $q=4$ and $1\leq
n\leq 6$ are given in Table~\ref{table_NOS_periods_III_q34}.  These values were derived using
Lemma~\ref{lemma:NOS-generaln-zerofree} and Table~\ref{table_k_q34}.

\begin{table}[htb]
\caption{$\mathcal{NOS}_q(n)$ periods (Construction III) for $q=3$ and $q=4$}
\label{table_NOS_periods_III_q34}
\begin{center}
\begin{tabular}{crrrrr} \hline
 ~    & $n=2$ & $n=3$ & $n=4$ & $n=5$ & $n=6$ \\ \hline
$q=3$ &    1  &    4  &    5  &    16 &   22 \\
$q=4$ &    3  &    10 &    31 &    96 &  294 \\  \hline
\end{tabular}
\end{center}
\end{table}

\subsection{Enumerating \texorpdfstring{$n$}{n}-tuples of given weight}
\label{section:enumerating_weight}

Lemmas~\ref{lemma:NOS-generaln-zerofree} and \ref{lemma:NOS-generaln-weight} mean that it is of
interest to know the value of $k_{q,n,w}$.  We have the following elementary result.

\begin{lemma}  \label{lemma:knqw}
As above, suppose $k_{q,n,w}$ ($q\geq2$, $n\geq 1$) is the number of zero-free $q$-ary $n$-tuples
with weight exactly $w$, where $k_{q,n,w}=0$ by definition if $w<n$ or $w>n(q-1)$.
\begin{enumerate}
\item[(i)] $k_{q,1,i}=1$, for $1\leq i\leq q-1$.
\item[(ii)] $\sum_{i=n}^{n(q-1)}k_{q,n,i} = (q-1)^n$;
\item[(iii)] if $n>1$, $k_{q,n,w} = \sum_{i=1}^{q-1}k_{q,n-1,w-i}$.
\item[(iv)] $k_{q,2,i}=q-1-|q-i|$, for $2\leq i\leq 2(q-1)$.
\item[(v)] if $q$ is even, $k_{q,3,3q/2}=\frac{3q^2-6q+4}{4}$.
\end{enumerate}
\end{lemma}

\begin{proof}
Parts (i) and (ii) are trivially true; (iii) follows immediately by considering the $q-1$ possible
ways of adding a final entry to a zero-free $(n-1)$-tuple to make a zero-free $n$-tuple of a given
weight. (iv) follows from (i) and (iii). Finally, since
\begin{align*}
k_{q,3,3q/2}&~~=~~\sum_{i=1}^{q-1}k_{q,2,3q/2-i}~~~\mbox{(from (iii))} \\
            &~~=~~\sum_{i=1}^{q-1} q-1-|i-q/2|~~~\mbox{(from (iv))}
\end{align*}
and the result follows. \qed
\end{proof}

\begin{corollary}  \label{corollary:kqnw}
Suppose $q>2$.
\begin{enumerate}
\item[(i)] If  $2\leq w\leq 2q-2$, then $k_{q,2,q}=q-1$ and $k_{q,2,w}<q-1$ if $w\neq
    q$.
\item[(ii)] If  $n\geq 3$ and $n\leq w\leq n(q-1)$, then $k_{q,n,w}<(q-1)^{n-1}$.
\end{enumerate}
\end{corollary}

\begin{proof}
\begin{itemize}

\item[(i)] This follows immediately by Lemma~\ref{lemma:knqw}(iv)).
\item[(ii)] By Lemma~\ref{lemma:knqw}(iii),
    \[ k_{q,3,w} = \sum_{i=1}^{q-1} {q,2,i} <(q-1)^2, \mbox{by (i)}.  \]
    The result now follows by induction using Lemma~\ref{lemma:knqw}(iii).
\end{itemize}
\end{proof}

The recursion of Lemma~\ref{lemma:knqw}(iii) makes calculating $k_{q,n,w}$ simple for small $n$ and
$q$. In fact these values are well-studied. It follows immediately from Lemma~\ref{lemma:knqw}(i)
and (iii) that the values $k_{3,n,w}$, $n\leq w\leq 2n$ are simply the binomial coefficients
$\binom{n}{w-n}$. More generally, using the notation of Freund \cite{Freund56},
$k_{q,n,w}=N_{q-2}(w-n,n)$, where $N_m(r,k)$ denotes the number of ways in which $r$ identical
objects can be distributed in $n$ cells with at most $m$ objects per cell. The values $k_{q,n,w}$
are simply the coefficients of the polynomial $(1+x+x^2+\cdots+x^{q-2})^n$. The values of
$N_m(r,k)$ are sometimes referred to as trinomial and quadrinomial coefficients in the cases $m=2$
and 3, and Comtet \cite[page 78]{Comtet74} tabulates the values of the trinomial and quadrinomial
coefficients for small $r$ and $k$.

\subsection{Examples of Construction III}

\subsubsection{Examples for \texorpdfstring{$q=3$}{q=3}}  \label{section:examples-NOS-zerofree}

For $q=3$ and small $n$, the zero-free low-weight tuples used in the construction of
Section~\ref{section:NOS-construction-3} are shown in Tables~\ref{table:0freetuples-n2q3} and
\ref{table:0freetuples-n3q3}.

\begin{table}[htb]
\caption{Zero-free $q$-ary $n$-tuples of weight less than $qn/2$: $n=2$, $q=3$}
\label{table:0freetuples-n2q3}
\begin{center}
\begin{tabular}{|l|c|}
\hline
$\mathbf{u}$ & $w(\mathbf{u})$  \\ \hline
11 & 2   \\ \hline
\end{tabular}
\end{center}
\end{table}

\begin{table}[htb]
\caption{Zero-free $q$-ary $n$-tuples of weight less than $qn/2$: $n=3$, $q=3$}
\label{table:0freetuples-n3q3}
\begin{center}
\begin{tabular}{|l|c||l|c||l|c||l|c|}
\hline
$\mathbf{u}$ & $w(\mathbf{u})$ & $\mathbf{u}$ & $w(\mathbf{u})$ & $\mathbf{u}$ & $w(\mathbf{u})$ & $\mathbf{u}$ & $w(\mathbf{u})$  \\ \hline
111 & 3   & 112 & 4   & 121 & 4  & 211 & 4 \\ \hline
\end{tabular}
\end{center}
\end{table}

The single 2-tuple in Table~\ref{table:0freetuples-n2q3} and the four 3-tuples in
Table~\ref{table:0freetuples-n3q3} give rise to a $\mathcal{NOS}_3(2)$ of period 1 with ring
sequence [1], and a $\mathcal{NOS}_3(3)$ of period 4 with ring sequence [1112], respectively.

\subsubsection{Examples for \texorpdfstring{$q=4$}{q=4}}

For $q=4$ and small $n$, the zero-free low-weight tuples are shown in
Tables~\ref{table:0freetuples-n2q4} and \ref{table:0freetuples-n3q4}.

\begin{table}[htb]
\caption{Zero-free $q$-ary $n$-tuples of weight less than $qn/2$: $n=2$, $q=4$}
\label{table:0freetuples-n2q4}
\begin{center}
\begin{tabular}{|l|c||l|c||l|c|}
\hline
$\mathbf{u}$ & $w(\mathbf{u})$ & $\mathbf{u}$ & $w(\mathbf{u})$ & $\mathbf{u}$ & $w(\mathbf{u})$  \\ \hline
11 & 2   & 12 & 3   & 21 & 3   \\ \hline
\end{tabular}
\end{center}
\end{table}

\begin{table}[htb]
\caption{Zero-free $q$-ary $n$-tuples of weight less than $qn/2$: $n=3$, $q=4$}
\label{table:0freetuples-n3q4}
\begin{center}
\begin{tabular}{|l|c||l|c||l|c||l|c||l|c|}
\hline
$\mathbf{u}$ & $w(\mathbf{u})$ & $\mathbf{u}$ & $w(\mathbf{u})$ & $\mathbf{u}$ & $w(\mathbf{u})$ & $\mathbf{u}$ & $w(\mathbf{u})$ & $\mathbf{u}$ & $w(\mathbf{u})$  \\ \hline
111 & 3   & 112 & 4   & 121 & 4  & 211 & 4  & 113 & 5   \\ \hline
131 & 5 & 311 & 5   & 122 & 5  & 212 & 5  & 221 & 5   \\ \hline
\end{tabular}
\end{center}
\end{table}

The three 2-tuples in Table~\ref{table:0freetuples-n2q4} and the ten 3-tuples in
Table~\ref{table:0freetuples-n3q4}  give rise to a $\mathcal{NOS}_4(2)$ of period 3 with ring
sequence [112], and a $\mathcal{NOS}_4(3)$ of period 10 with ring sequence [1113112212],
respectively.

\section{Orientable sequences}  \label{section:OS_near_optimal}

We now describe a way of using the negative orientable sequences generated by Construction II to
yield orientable sequences for general $n$ and $q$ with period  of the same order  of magnitude as
the  upper bound.  We first need to introduce the Lempel Homorphism, which is fundamental to
the construction of orientable sequences.

\subsection{The de Bruijn graph and the Lempel Homomorphism}

Following Alhakim et al.\ \cite{Alhakim24a} we also introduce the de Bruijn graph. For positive
integers $n$ and $q$ greater than one, let $\mathbb{Z}_q^n$ be the set of all $q^n$ vectors of
length $n$ with entries from the group $\mathbb{Z}_q$ of residues modulo $q$. A de Bruijn sequence
\emph{of order $n$} with alphabet in $\mathbb{Z}_q$ is a periodic sequence that includes every
possible $n$-tuple precisely once as a subsequence of consecutive symbols in one period of the
sequence.

The order $n$ de Bruijn digraph, $B_n(q)$, is a directed graph with $\mathbb{Z}^n_q$ as its vertex
set and where, for any two vectors $\textbf{x} = (x_1,x_2,\ldots,x_n)$ and $\textbf{y} =
(y_1,y_2,\ldots,y_n), \ (\textbf{x}; \textbf{y})$ is an edge if and only if $y_i = x_{i+1}$ for
every $i$ ($1\leq i< n$). We then say that $\textbf{x}$ is a \emph{predecessor} of $\textbf{y}$ and
$\textbf{y}$ is a \emph{successor} of $\textbf{x}$. Evidently, every vertex has exactly $q$
successors and $q$ predecessors.

A cycle in $B_n(q)$ is a path that starts and ends at the same vertex. It is said to be
\emph{vertex disjoint} if it does not visit any vertex more than once. Two cycles or two paths in
the digraph are vertex-disjoint if they do not have a common vertex.

Following the notation of Lempel~\cite{Lempel70}, a convenient representation of a vertex disjoint
cycle $(\textbf{x}^{(1)}; \dots ; \textbf{x}^{(l)})$ is the \emph{ring sequence} $[x^1, \dots , x^l
]$ of symbols from $\mathbb{Z}_q$ defined such that the $i$th vertex in the cycle starts with the
symbol $x^i$. Clearly a vertex disjoint cycle with ring sequence $[x_1,x_2,\ldots,x_m]$ corresponds
to an $n$-window sequence of period $m$ where $s_i=x_j$ whenever $i\equiv j \pmod m$.


Finally, we need a well-established generalisation of the Lempel graph homomorphism~\cite{Lempel70}
to non-binary alphabets --- see, for example, Alhakim and Akinwande~\cite{Alhakim11}.

\begin{definition} \label{Lempel}
For a nonzero $\beta \in \mathbb{Z}_q$, we define a function $D_{\beta}$ from $B_n(q)$ to
$B_{n-1}(q)$ as follows. For $\textbf{a} = (a_1, \dots , a_n)$ and $\textbf{b} = (b_1, \dots ,
b_{n-1}), \ D_{\beta}(\textbf{a}) = \textbf{b}$ if and only if $b_i = d_{\beta}(a_i , a_{i+1})$ for
$i = 1$ to $n-1$, where $d_{\beta}(a_i , a_{i+1}) = \beta(a_{i+1} - a_i) \mod q.$
\end{definition}

We extend the notation to allow the Lempel morphism $D_{\beta}$ to be applied to periodic sequences
in the natural way, as we now describe. That is, $D_{\beta}$ (where $\beta\in \mathbb{Z}_q$) is the
map from the set of periodic sequences to itself defined by
\[ D((s_i))= \{(t_i): t_j=\beta(s_{j+1}-s_{j}) \}. \]
The image of a sequence of period $m$ will clearly have period dividing $m$. In the usual way we
can define $D_{\beta}^{-1}$ to be the \emph{inverse} of $D_{\beta}$, i.e.\ if $S$ is a periodic
sequence than $D_{\beta}^{-1}(S)$ is the set of all sequences $T$ with the property that
$D_{\beta}(T) = S$.

We are particularly interested in the case $\beta=1$, and we simply write $D$ for $D_1$. Examples of the application of
$D^{-1}$ are given in Sections~\ref{subsection_OS_from_NOS_examples_1} and
\ref{subsection_OS_from_NOS_examples_2}.

The
\emph{weight} $w(S)$ of a cycle $S$ is the weight of the ring sequence corresponding to $S$ (that
is the sum of the terms $s_0,s_1,\dots,s_{m-1}$ treating $s_i$ as an integer in the range
$[0,q-1]$). Similarly we write $w_q(S)$ for $w(S) \bmod{q}$.

\subsection{Construction methods}

We need the following key result from Alhakim et al. \cite{Alhakim24a}.

\begin{result}[Theorem 6.10 of \cite{Alhakim24a}]  \label{result:Lempel-negative-orientable}
Suppose $S=(s_i)$ is a negative orientable sequence of order $n$ and period $m$. If $w_q(S)$ has
additive order $h$ as a residue modulo $q$ then $D^{-1}(S)$ consists of $h$ shifts of each of $q/h$
mutually no-disjoint orientable sequences of order $n+1$ and period $hm$ which are translates of
one another.
\end{result}

So, in particular, if a $\mathcal{NOS}_q(n)$ $S$ has period $m$ and $(w_q(S),q)=1$ (i.e.\ $w_q(S)$
is a unit in $\mathbb{Z}_q$) then $D^{-1}(S)$ is an $\mathcal{OS}_q(n+1)$ of period $qm$. We
propose to employ the negative orientable sequences described in Sections~\ref{section:NOSq2} and
\ref{section:NOS-generaln} to generate orientable sequences using this approach. Unfortunately, we
do not know in general which of the negative orientable sequences have weight a unit in
$\mathbb{Z}_q$; however, regardless of their actual weight, we next describe a very simple way of
modifying these sequences to achieve the desired result.

We first make a simple observation regarding units in $\mathbb{Z}_q$ for any $q$ (where a unit
$u\in \mathbb{Z}_q$ is any value such that $(u,q)=1$).

\begin{lemma} \label{lemma:units}
Suppose $q>2$ ($q\neq6$) and $w\in\mathbb{Z}_q$ is a non-unit, i.e.\ $(w,q)\neq1$.  Then there
exists $d\in\{1,2,\ldots,\lfloor (q-1)/2\rfloor\}$ such that $w-d$ is a unit.
\end{lemma}

\begin{proof}
If $q$ is prime then this follows immediately since every non-zero element in $\mathbb{Z}_q$ is a
unit. If $q=4$, then $w$ must be 0 or 2 and in both cases $w-1$ is a unit.  Hence suppose $q>6$.

We first show that there exists a unit $u$ satisfying $q/2< u\leq 3q/4$. If $q$
is odd, $q=2s+1$ say ($s\geq3$), then trivially $s+1$ is a unit in
$\mathbb{Z}_{2s+1}$ since $(s+1,2s+1)=1$ for every $s>0$, and $s+1$ is within
the desired range. If $q$ is even, first suppose $q\equiv0\pmod 4$, i.e.\
$q=4s$ for some $s\geq2$.  Then $2s+1$ is a unit in $\mathbb{Z}_{4s}$ since
$(2s+1,4s)=1$ for every $s>0$, and $2s+1$ is within the desired range. Finally
suppose $q\equiv2\pmod 4$, i.e.\ $q=4s+2$ for some $s\geq2$. Then $2s+3$ is a
unit in $\mathbb{Z}_{4s+2}$ since $(2s+3,4s+2)=1$ for every $s>0$, and $2s+3$
is within the desired range since $s\geq 2$.

If $w=0$ then $w-1$ is a unit and if $0<w\leq q/2$ then clearly there exists
$d\in\{1,2,\ldots,\lfloor (q-1)/2\rfloor\}$ such that $w-d=1$, and the result follows. Hence
suppose $q/2<w<q-1$.  Since $q>6$ there exists a unit $u$ satisfying $q/2<u\leq 3q/4$. If $u<w$
then clearly there will exist a $d\in\{1,2,\ldots,\lfloor (q-1)/2\rfloor\}$ such that $w-d=u$. If
$w<u$ then there exists a $d\in\{1,2,\ldots,\lfloor (q-1)/2\rfloor\}$ such that $w-d=-u$; this
follows since $w<u\leq 3q/4$ and $-u\geq q/4$.  The result follows by noting that $-u$ must be a
unit. \qed
\end{proof}

As is well-known, it is possible to remove a uniform $n$-tuple from a sequence
without changing the other $n$-tuples present, simply by deleting a single
occurrence of a symbol from within the uniform $n$-tuple.  That is, if $S$ is
a $\mathcal{NOS}_q(n)$ of period $m$ containing the $i$-uniform $n$-tuple
($i,i,i\ldots,i$), deleting a single occurrence of $i$ from within this tuple
will result in a $\mathcal{NOS}_q(n)$ $S'$ of period $m-1$, where
$w_q(S')=w_q(S)-i$.  This leads to the following simple construction.

\begin{construction}  \label{construction:NOS-unit-weight}
Suppose $q>2$ and let $S$ be a $\mathcal{NOS}_q(n)$ with the property that it contains all the
$i$-uniform $n$-tuples ($i,i,i\ldots,i$) for every $i$ satisfying $1\leq i<q/2$.  We derive a sequence
$S'$ from $S$ in the following way.
\begin{itemize}
\item If $w_q(S)$ is a unit in $\mathbb{Z}_q$ then set $S'=S$.
\item If $w_q(S)$ is not a unit in $\mathbb{Z}_q$ and $q\neq6$, choose $i\in[1,q/2)$ such that
    $w_q(S)-i$ is a unit ($i$ exists by Lemma~\ref{lemma:units}). By assumption, $S$ contains
    the $i$-uniform $n$-tuple ($i,i,\ldots,i$), and let $S'$ be derived from $S$ by deleting a
    single occurrence of $i$ from this $n$-tuple.
\item If $w_q(S)$ is not a unit in $\mathbb{Z}_q$ and $q=6$, then $w_q(S)$ must be one of 0, 2,
    3 and 4. Observe that $S$ must contain the uniform $n$-tuples $(1,1,\ldots,1)$ and
    $(2,2,\ldots,2)$.  If $w_q(S)=0$ or 2, let $S'$ be derived from $S$ by deleting a single
    occurrence of $1$ from the 1-uniform $n$-tuple $(1,1,\ldots,1)$. If $w_q(S)=3$, let $S'$ be
    derived from $S$ by deleting a single occurrence of $2$ from $(2,2,\ldots,2)$. If
    $w_q(S)=4$ let $S'$ be derived from $S$ by deleting a single occurrence of $1$ from
    $(1,1,\ldots,1)$ and a single occurrence of $2$ from $(2,2,\ldots,2)$.
\end{itemize}
\end{construction}

\begin{lemma}  \label{lemma:NOS-unit-weight}
Suppose that $q>2$, $S$ is a $\mathcal{NOS}_q(n)$ of period $m$ containing all the $i$-uniform
$n$-tuples ($i,i,i\ldots,i$) for every $i$ satisfying $1\leq i<q/2$, and $S'$ is derived from $S$
using Construction~\ref{construction:NOS-unit-weight}. Then $S'$ is a $\mathcal{NOS}_q(n)$ with
$w_q(S')$ a unit in $\mathbb{Z}_q$, and period at least $m-1$ ($q\neq6$) or at least $m-2$ ($q=6$).
\end{lemma}

\begin{proof}
Observe that the construction of $S'$ in the case $q=6$ ensures that $w_6(S')$ is a unit. The
result follows trivially by the observation that removing a single element from a uniform
$n$-tuple will not affect negative orientability. \qed
\end{proof}

\subsection{Orientable sequences for \texorpdfstring{$n=3$}{n=3}}

We can now give a method for constructing near-optimal orientable sequences for $n=3$ and any
$q>2$. These have period asymptotical of the same order as the maximum possible as $q\
\rightarrow\infty$.

\begin{construction}  \label{construction:OS3}
Suppose $q>2$ and let $S$ be a maximal $\mathcal{NOS}_q(2)$ constructed using the method described
in Section~\ref{section:NOSq2}. Let $S'$ be derived from $S$ using
Construction~\ref{construction:NOS-unit-weight} (observing that $S$ satisfies the conditions of the
construction by Note~\ref{note:constant2}). Then let $T$ be derived from $S'$ using
Theorem~\ref{result:Lempel-negative-orientable}.
\end{construction}

\begin{theorem}  \label{theorem:OS3}
Suppose $q>2$ and suppose that $T$ is constructed from $S$ using
Construction~\ref{construction:OS3}. Then $T$ is an $\mathcal{OS}_q(3)$ of period at least:
\begin{align*}
q\left(\frac{q(q-1)}{2}-1\right) ~~&~\mbox{if $q$ is odd;}\\
q\left(\frac{q(q-1)}{2}-2\right) ~~&~\mbox{if $q$ is even ($q\not=6$);}\\
q\left(\frac{q(q-1)}{2}-3\right) ~~&~\mbox{if $q=6$.}
\end{align*}
\end{theorem}

\begin{proof}
By the discussions in Section~\ref{section:NOSq2}, $S$ has period either $q(q-1)/2$ if $q$ is odd
or $q(q-1)/2-1$ if $q$ is even. Hence, by Lemma~\ref{lemma:NOS-unit-weight} $S'$ has period at
least:
\begin{itemize}
\item $q(q-1)/2-1$ if $q$ is odd;
\item $q(q-1)/2-2$ if $q$ is even;
\item $q(q-1)/2-3$ if $q=6$.
\end{itemize}
Applying Theorem~\ref{result:Lempel-negative-orientable} multiplies the period by $q$, since in
every case $w_q(S')$ is a unit, and the result follows. \qed
\end{proof}
Periods of the orientable sequences obtained using Theorem~\ref{theorem:OS3} are tabulated for
small $q$ in Table~\ref{table_OS_periods_n3}.  The value of the upper bound given in
\cite{Alhakim24a} is given in brackets for comparison purposes.

\begin{table}[htb]
\caption{$\mathcal{OS}_q(3)$ periods (and bounds)} \label{table_OS_periods_n3}
\begin{center}
\begin{tabular}{ccccccc} \hline
$q=3$   & $q=4$   & $q=5$   & $q=6$   & $q=7$     & $q=8$     \\ \hline
  6 (9) & 16 (22) & 45 (50) & 72 (87) & 140 (147) & 208 (220) \\  \hline
\end{tabular}
\end{center}
\end{table}

\subsection{Orientable sequences for general \texorpdfstring{$n$}{n}}

We can apply an almost identical approach to the negative orientable sequences
described in Section~\ref{section:NOS-generaln}. First observe that if a
$\mathcal{NOS}_q(n)$ is constructed as given in
Section~\ref{section:NOS-generaln}, then it will contain the $i$-uniform
$n$-tuples ($i,i,i\ldots,i$) for every $i$ satisfying $1\leq i<q/2$.

\begin{construction}  \label{construction:OSn}
Suppose $q>2$ and $n>2$, and let $S$ be a $\mathcal{NOS}_q(n-1)$ constructed
using the method described in Section~\ref{section:NOS-generaln}. Let $S'$ be
derived from $S$ using Construction~\ref{construction:NOS-unit-weight}
(observing that $S$ satisfies the conditions of the construction as noted
above). Then let $T$ be derived from $S'$ using
Result~\ref{result:Lempel-negative-orientable}.
\end{construction}

\begin{theorem}  \label{theorem:orientable-1}
Suppose $n>2$, $q>2$, and $T$ is derived using
Construction~\ref{construction:OSn}. Then $T$ is an $\mathcal{OS}_q(n)$ of
period at least:
\begin{align*}
q(q^{n-1}-r_{q,n-1,(n-1)q/2}-2)/2~~&~\mbox{if $q\neq6$};\\
q(q^{n-1}-r_{q,n-1,(n-1)q/2}-4)/2~~&~\mbox{if $q=6$;}
\end{align*}
where, as previously, $r_{n-1,q,s}$ is the number of $q$-ary $(n-1)$-tuples
with pseudoweight exactly $s$.
\end{theorem}

\begin{proof}
By Lemma~\ref{lemma:NOS-generaln} there exists a $\mathcal{NOS}_q(n-1)$ $S$ of period
$(q^{n-1}-r_{n,q-1})/2$.  By Lemma~\ref{lemma:NOS-unit-weight}, constructing $S'$ from $S$ using
Construction~\ref{construction:NOS-unit-weight} will yield a $\mathcal{NOS}_q(n-1)$ of period:
\begin{itemize}
\item $\frac{q^{n-1}-r_{q,n-1,(n-1)q/2}}{2}-1$ if $q\neq6$; and
\item $\frac{q^{n-1}-r_{q,n-1,(n-1)q/2}}{2}-2$ if $q=6$.
\end{itemize}
The result follows from Result~\ref{result:Lempel-negative-orientable}. \qed
\end{proof}

By Note~\ref{note:rqns}, these sequences have period of the same order of magnitude as the
upper bound.
Periods of the orientable sequences obtained using Theorem~\ref{theorem:orientable-1} are tabulated
for small $q$ in Table~\ref{table_OS_periods_general_n}.  As previously, the value of the upper
bound given in \cite{Alhakim24a} is given in brackets for comparison purposes.

\begin{table}[htb]
\caption{$\mathcal{OS}_q(n)$ periods (and bounds)} \label{table_OS_periods_general_n}
\begin{center}
\begin{tabular}{c|rrrrrr} \hline
$n$ & $q=3$   & $q=4$    & $q=5$    & $q=6$    & $q=7$     & $q=8$     \\ \hline
3   & 6       & 16       & 45       & 72       & 140       & 208       \\
    & (9)     & (22)     & (50)     & (87)     & (147)     & (220)     \\ \hline
4   & 27      & 84       & 275      & 522      & 1127      & 1792      \\
    & (33)    & (118)    & (290)    & (627)    & (1155)    & (2012)    \\ \hline
5   & 90      & 368      & 1385     & 3288     & 7756      & 14672     \\
    & (105)   & (478)    & (1490)   & (3777)   & (8211)    & (16124)   \\ \hline
6   & 285     & 1540     & 7155     & 20160    & 56049     & 118856    \\
    & (336)   & (2014)   & (7680)   & (23217)  & (58464)   & (130812)  \\ \hline
7   & 879     & 6340     & 35805    & 122634   & 388619    & 959152    \\
    & (1032)  & (8062)   & (38640)  & (139317) & (410256)  & (1046524) \\ \hline
8   & 2688    & 25900    & 181095   & 743358   & 2757930   & 7724544   \\
    & (3189)  & (32638)  & (194630) & (839157) & (2879835) & (8386556) \\ \hline
\end{tabular}
\end{center}
\end{table}
\subsection{Examples}  \label{subsection_OS_from_NOS_examples_1}

We next show how the three examples of negative orientable sequences given in
Section~\ref{section:NOS-examples} can be used in the above construction method
to yield orientable sequences.

As noted above, the $\mathcal{NOS}_3(2)$ $S_{3,2}=[011]$ has $w_3(S_{3,2})=2$, a unit in
$\mathbb{Z}_3$ and hence can be used directly in Result~\ref{result:Lempel-negative-orientable}.
Applying $D^{-1}$ to $S_{3,2}$ results in the $\mathcal{OS}_3(3)$ with ring sequence $[001220112]$,
of period 9.  This is optimal (see \cite[Table 1]{Alhakim24a}).

The $\mathcal{NOS}_3(3)$ $S_{3,3}=[0100112111]$ has $w_3(S_{3,3})=2$, which is
again a unit in $\mathbb{Z}_3$ and hence it too can be used directly in
Result~\ref{result:Lempel-negative-orientable}. Applying $D^{-1}$ to $S_{3,3}$
results in the $\mathcal{OS}_3(4)$ with ring sequence
\[ [0011120201~2200012120~1122201012], \]
of period 27.  This is six less than the bound of \cite[Theorem 4.11]{Alhakim24a}.

Finally, the $\mathcal{NOS}_4(3)$ $S_{4,3}$ has $w_4(S_{4,3})=0$, which is not a unit in
$\mathbb{Z}_4$.  To use it in Result~\ref{result:Lempel-negative-orientable} we need to remove a
single `1' from the 1-uniform 3-tuple (111), resulting in the $\mathcal{NOS}_4(3)$
\[ S'_{4,3}=[0122~1201~0210~0113~1121~1]\]
with period 21 and $w_4(S'_{4,3})=3$, a unit in $\mathbb{Z}_4$. Applying $D^{-1}$ to $S'_{4,3}$
results in the $\mathcal{OS}_4(4)$ with ring sequence:
\begin{align*} [0013~1200~1130~0012~1231~2~~3302~0133~0023~3301~0120~1\\
2231~3022~3312~2230~3013~0~~1120~2311~2201~1123~2302~3 ],
\end{align*}
of period 84. This is 34 less then the bound of \cite[Theorem 4.11]{Alhakim24a}.

\section{More orientable sequences}  \label{section:goodOS}

In this final main section we describe how to construct orientable sequences using a second
approach described by Alhakim et al.\ \cite{Alhakim24a}. These have period of the same order of magnitude as the
upper bound.

\subsection{Preliminaries}

We first need the following definitions.

\begin{definition}[Definition 6.16 of \cite{Alhakim24a}]  \label{definition:runs}
Let $S=(s_i)$ be a periodic $q$-ary sequence. For $a \in \mathbb{Z}_q$ we write $a^t$ for a string
of $t$ consecutive terms $a$ of $S$. A \emph{run} of $a$ in $S$ is a string
$s_j,\dots,s_{j+t-1}=a^t$ with $s_{j-1},s_{j+t} \not=a$. A run, $a^t$, is \emph{maximal} in $S$ if
any string $a^{t^{\prime}}$ of $S$ has $t^{\prime} \le t$  and further is \emph{inverse maximal}
if, in addition, any string $1^{t^{\prime}} + (q-1-a)^{t^{\prime}}$ of $S$ has $t^{\prime} \le t$.
\end{definition}

We note that if $a^t$ is a maximal run of $S$ then $(a + b)^t$ is a maximal run of the translate by
$b$ of $S$ and there exists a shift of such a translate that has ring sequence whose first $t$
terms are $a + b$. Moreover if $S$ is orientable (respectively negative orientable) then so is this
shifted translate.

\begin{definition}[Definition 6.17 of \cite{Alhakim24a}]  \label{definition:Ea}
Let $ a \in \mathbb{Z}_q$. Suppose that the ring sequence of a periodic sequence $S$ is
$[s_0,s_1,\cdots,s_{m-1}]$ and $r$ is the smallest non-negative integer such that
$a^t=s_r,s_{r+1},\cdots,s_{r+t-1}$ is a maximal run for $a \in \mathbb{Z}_q$. Define the sequence
$\mathcal{E}_a(S)$ to be the sequence with ring sequence
 \[ [s_0,s_1,\cdots,s_{r-1},a,s_r,s_{r+1},\cdots,s_{m-1}] \]
i.e.\ where the occurrence of $a^t$ is replaced with $a^{t+1}$.
\end{definition}

We also need:

\begin{definition}[Definition 6.20 of \cite{Alhakim24a}]  \label{definition:good}
An orientable (respectively negative orientable) sequence with the property that any run of $0$ has
length at most $n-2$ is said to be \emph{good}.
\end{definition}

We can now give the following result.

\begin{result}  \label{result:Cor-length1}[Corollary 6.22 of \cite{Alhakim24a}]  \label{result:recursive}
Suppose $S_n$ is a good orientable (respectively negative orientable) sequence of order $n$ and
period $m_n$ with the property that $w_q(S)$ is a unit in $\mathbb{Z}_q$ ($n\geq 2$ and $q>2$).
Recursively define the sequences $S_{i+1}=\mathcal{E}_a(D^{-1}(S_i))$, where $a =
1-w_q(D^{-1}(S_i))$, for $i\geq n$, and suppose $S_i$ has period $m_i$ ($i>n$). Then, $S_{i}$ is a
good orientable or negative orientable sequence for every $i\ge n$, and $m_{n+j+1}=qm_{n+j}+1$ for
every $j\geq 0$. $S_{i}$ is an orientable (respectively negative orientable) sequence when $i-n$ is
even and a negative orientable (respectively orientable) sequence when $i-n$ is odd.
\end{result}

\begin{note}
Corollary 6.22 of \cite{Alhakim24a} differs slightly from the above statement in that it makes the
slightly more restrictive assumption that $w_q(S)=1$. However, it is straightforward to verify that
the result holds whenever $w_q(S)$ is a unit in $\mathbb{Z}_q$.
\end{note}

The following simple result follows.

\begin{corollary}
In the notation of Result~\ref{result:Cor-length1}
\[ m_{n+s} = q^{s}m_n + \frac{(q^{s}-1)}{(q-1)} \]
for every $s\geq 1$.
\end{corollary}

\begin{proof}
We establish the result by induction on $s$. Suppose $s=1$; then from
Result~\ref{result:Cor-length1} we immediately know that
\[ m_{n+1} = qm_n+1 \]
i.e.\ the desired result holds in this case since, for $s=1$:
\[ q^{s}m_n + \frac{(q^{s}-1)}{(q-1)} = qm_n + \frac{(q-1)}{(q-1)} = qm_n+1. \]

Now suppose the results holds for $s=t>0$. That is, suppose
\[ m_{n+t} = q^{t}m_n + \frac{(q^{t}-1)}{(q-1)}. \]
By Result~\ref{result:Cor-length1},
\[ m_{n+t+1} = q(q^{t}m_n + \frac{(q^{t}-1)}{(q-1)})+1 = q^{t+1}m_n + \frac{(q^{t+1}-q)}{(q-1)}) + 1\]
and the result clearly holds for $t+1$. \qed
\end{proof}

This immediately gives.

\begin{corollary}  \label{corollary:period}
Suppose $S_n$ is a good $\mathcal{NOS}_q(n)$ of period $m$ with the property that $w_q(S)$ is a
unit in $\mathbb{Z}_q$.  Then $S_{n+2s+1}$ is a good $\mathcal{OS}_q(n+2s+1)$ of period $\ell_{s}$
for every $s>0$, where
\[ \ell_{s} = q^{2s+1}m + \frac{(q^{2s+1}-1)}{(q-1)}                 .\]
\end{corollary}

\subsection{Recursively constructing orientable sequences}

Combining Result~\ref{result:recursive} and Corollary~\ref{corollary:period} with Construction III,
as given in Section~\ref{section:NOS-construction-3}, yields the following result.

\begin{theorem} \label{theorem:orientable-2}
If $q>2$, there exists a:
\begin{itemize}
\item good $\mathcal{OS}_q(2s+3)$ of period at least
\[ q^{2s+1}m_2 + \frac{(q^{2s+1}-1)}{(q-1)}  \]
\item good $\mathcal{OS}_q(2s+4)$ of period at least
\[ q^{2s+1}m_3 + \frac{(q^{2s+1}-1)}{(q-1)} \]
\end{itemize}
for every $s\geq 0$, where

\[
m_2 =
\begin{cases}
 (q-1)(q-2)/2-1, & \text{if~} q\neq6 \\
 (q-1)(q-2)/2-2, & \text{if~} q=6
\end{cases}
\]

and

\[
m_3 =
\begin{cases}
(q-1)^3/2-1,                  & \text{if $q$ is odd} \\
((q-1)^3- (3q^2-6q+4)/4)/2-1, & \text{if $q$ is even ($q\neq6$)} \\
((q-1)^3- (3q^2-6q+4)/4)/2-2, & \text{if $q=6$}.
\end{cases}
\]

\end{theorem}

\begin{proof}
First observe that if $S$ is a negative orientable sequence obtained using the method underlying
Lemma~\ref{lemma:NOS-generaln-zerofree} then it is good by definition (since it contains no zeros).
It also clearly contains all the $i$-uniform $n$-tuples $(i,i,\ldots,i)$ for every $i$ satisfying
$1\leq i<q/2$. If it has period $m$ then, using Construction~\ref{construction:NOS-unit-weight}, we
can construct a good $\mathcal{NOS}_q(n)$ $S'$ where $w(S')$ is a unit in $\mathbb{Z}_q$, with
period at least $m-1$ (or $m-2$ if $q=6$).
\begin{itemize}
\item From Lemma~\ref{lemma:NOS-generaln-zerofree} with $n=2$ there exists a good
    $\mathcal{NOS}_q(2)$ of period $((q-1)^2- (q-1))/2=(q-1)(q-2)/2$ (observing that
    $k_{q,2,q}=q-1$, from Lemma~\ref{lemma:knqw}(iv)). Hence from the above discussion there
    exists a good $\mathcal{NOS}_q(2)$ whose weight is a unit in $\mathbb{Z}_q$ with period at
    least $(q-1)(q-2)/2-1$ if $q\neq 6$ and period at least $(q-1)(q-2)/2-2$ if $q=6$.

\item From Lemma~\ref{lemma:NOS-generaln-zerofree} with $n=3$ there exists a good
    $\mathcal{NOS}_q(3)$ of period $(q-1)^3/2$ if $q$ is odd and period $((q-1)^3-
    k_{q,3,3q/2})/2$ if $q$ is odd. Hence from the above discussion there exists a good
    $\mathcal{NOS}_q(3)$ whose weight is a unit in $\mathbb{Z}_q$ with period at least
    $(q-1)^3/2-1$ if $q$ is odd, $((q-1)^3- k_{q,3,3q/2})/2-1$ if $q$ is even ($q\neq 6$) and
    period at least $((q-1)^3- k_{q,3,3q/2})/2-2$ if $q=6$.  The result follows by noting that
    $k_{q,3,3q/2})/2=(3q^2-6q+4)/4$, from Lemma~\ref{lemma:knqw}(v). \qed

\end{itemize}
\end{proof}

Thus if, for example, $q$ is odd, there exists a good $\mathcal{OS}_q(2s+4)$ of period at least
\[  q^{2s+1}(q-1)^3/2 - 1 + \frac{(q^{2s+1}-1)}{(q-1)}. \]
For each $s$ these have period of the same order  of magnitude as the
upper bound.

\subsection{Examples}  \label{subsection_OS_from_NOS_examples_2}

We conclude by giving two simple examples of the approach described.

\begin{itemize}
\item {\bf Generating a good $\mathcal{OS}_3(2i)$ for every $i\geq2$}

From Section~\ref{section:examples-NOS-zerofree} we know that $S_3=[1112]$ is a good
$\mathcal{NOS}_3(3)$ of period 4 where $w_3(S_3)=2$, a unit in $\mathbb{Z}_3$. Applying
Result~\ref{result:recursive},
\[ D^{-1}(S_3) = [012020121201] \]
and
\[ S_4 = \mathcal{E}_1(D^{-1}(S_3)) = [0120201212011] \]
is a good $\mathcal{OS}_3(4)$ of period 13 and $\mathbb{Z}_3$-weight 1. This recursive process
can now be repeated to obtain a good $\mathcal{OS}_3(2i)$ with unit $\mathbb{Z}_3$-weight for
every $i>1$.

\item {\bf Generating a good $\mathcal{OS}_4(2i+1)$ for every $i\geq 1$}

From Section~\ref{section:examples-NOS-zerofree} we know that $[112]$ is a good
$\mathcal{NOS}_4(2)$ of period 3 with $\mathbb{Z}_4$-weight 0. We can delete a single 1 to
obtain $T_2=[12]$, a good $\mathcal{NOS}_4(2)$ of period 2 with $\mathbb{Z}_4$-weight 3, a unit
in $\mathbb{Z}_4$. Applying Result~\ref{result:recursive},
\[ D^{-1}(T_2) = [01302312] \]
and
\[ T_3 = \mathcal{E}_1(D^{-1}(T_2)) = [011302312] \]
is a good $\mathcal{OS}_4(3)$ of period 9 and $\mathbb{Z}_4$-weight 1. This recursive process
can now be repeated to obtain a good $\mathcal{OS}_4(2i+1)$ with unit $\mathbb{Z}_4$-weight for
every $i>0$.

\end{itemize}

\section{Concluding remarks}  \label{section:Conclusions}

In this paper we used two of the approaches proposed in \cite{Alhakim24a} to generate orientable
sequences with period of the same order of magnitude as the
upper bound. Whilst the second
approach yields sequences with somewhat shorter periods, in practical applications it may be that
its recursive (and explicit) approach is more convenient.

There remain a variety of directions for future research. The constructions in this paper require
an exponential amount of space with respect to $n$. They either require storage of a graph of
exponential size to use an Eulerian circuit algorithm, or storage of the sequence of a smaller
order to use Lempel's homomorphism. It would be desirable to be able to generate negative
orientable sequences more space-efficiently. Another potential future research direction would be
to generate longer negative orientable sequences by incorporating some words with pseudoweight
equal to $nq/2$.  Finally, it remains to consider how to devise potential input sequences for the
other methods of construction proposed in \cite{Alhakim24a}.

\providecommand{\bysame}{\leavevmode\hbox to3em{\hrulefill}\thinspace}
\providecommand{\MR}{\relax\ifhmode\unskip\space\fi MR }
\providecommand{\MRhref}[2]{%
  \href{http://www.ams.org/mathscinet-getitem?mr=#1}{#2}
} \providecommand{\href}[2]{#2}

\end{document}